\documentclass[12pt]{article}
\usepackage{amsmath}
\usepackage{amssymb}

\usepackage{amsthm}
\usepackage[dvipsnames]{pstricks}

\newtheorem{theorem}{Theorem}[section]
\newtheorem{lemma}[theorem]{Lemma}
\newtheorem{proposition}[theorem]{Proposition}
\newtheorem{corollary}[theorem]{Corollary}

{\theoremstyle{definition}}
{\theoremstyle{definition}\newtheorem{definition}[theorem]{Definition}}
{\theoremstyle{definition}}

\numberwithin{equation}{section}

\begin{document}

\def\C{{\mathbb C}}
\def\N{{\mathbb N}}
\def\Z{{\mathbb Z}}
\def\R{{\mathbb R}}
\def\Q{{\mathbb Q}}
\def\K{{\mathbb K}}
\def\F{{\mathcal F}}
\def\H{{\mathcal H}}
\def\E{{\mathcal E}}
\def\B{{\mathcal B}}
\def\rr{{\mathcal R}}
\def\pp{{\mathcal P}}

\def\d{ \{\!\{ }
\def\c{ \}\!\} }
\def\Hom{\rm Hom}
\def\epsilon{\varepsilon}
\def\kappa{\varkappa}
\def\phi{\varphi}
\def\leq{\leqslant}
\def\geq{\geqslant}
\def\slim{\mathop{\hbox{$\overline{\hbox{\rm lim}}$}}\limits}
\def\ilim{\mathop{\hbox{$\underline{\hbox{\rm lim}}$}}\limits}
\def\supp{\hbox{\tt supp}\,}
\def\dim{{\rm dim}\,}
\def\ssub#1#2{#1_{{}_{{\scriptstyle #2}}}}
\def\dimk{{\ssub{\dim}{\K}\,}}
\def\ker{\hbox{\tt ker}\,}
\def\im{\hbox{\tt im}\,}
\def\spann{\hbox{\tt span}\,}
\def\supp{\hbox{\tt supp}\,}
\def\deg{\hbox{\tt \rm deg}\,}
\def\bin#1#2{\left({{#1}\atop {#2}}\right)}
\def\summ{\sum\limits}
\def\maxx{\max\limits}
\def\minn{\min\limits}
\def\limm{\lim\limits}
\def\ootimes{\,{\text{$\scriptstyle\otimes$}}\,}
\def\oo{\otimes}

\def\k{\mathbb K}
\def\lll{\langle}
\def\rrr{\rangle}
\def\dd{\delta}
\def\a{\alpha}
\def\b{\beta}
\def\g{\gamma}
\def\ai{$A_{\infty}$}
\def\l{\langle}
\def\r{\rangle}
\def\lor{k\langle\langle x,y \rangle\rangle}
\def\xd#1{\frac{\partial}{\partial x_{#1}}}

\newcommand{\ff}{\footnote}

\def\C{{\mathbb C}}
\def\N{{\mathbb N}}
\def\Z{{\mathbb Z}}
\def\R{{\mathbb R}}
\def\PP{\cal P}
\def\p{\rho}
\def\phi{\varphi}
\def\ee{\epsilon}
\def\ll{\lambda}
\def\a{\alpha}
\def\bb{\beta}
\def\D{\Delta}
\def\dd{\delta}
\def\g{\gamma}
\def\rk{\text{\rm rk}\,}
\def\dim{\text{\rm dim}\,}
\def\ker{\text{\rm ker}\,}
\def\square{\vrule height6pt width6pt depth 0pt}
\def\epsilon{\varepsilon}
\def\phi{\varphi}
\def\kappa{\varkappa}
\def\strl#1{\mathop{\hbox{$\,\leftarrow\,$}}\limits^{#1}}

\vskip1cm


\title{Pre-Calabi-Yau algebras
  and double Poisson brackets}


\author{Natalia Iyudu, Maxim Kontsevich, Yannis Vlassopoulos}



\date{}

\maketitle

\begin{abstract}


We give an explicit formula
showing how the double Poisson algebra introduced in \cite{VdB}  appears as a particular part of a pre-Calabi-Yau structure, i.e. cyclically invariant, with respect to the natural inner form, solution of the Maurer-Cartan equation on $A\oplus A^*$. Specific part of this solution is described, which is in one-to-one correspondence with the double Poisson algebra structures. The result holds for any associative algebra $A$ and emphasizes the  special role  of the fourth component of a pre-Calabi-Yau structure in this respect.
As a consequence we have that appropriate pre-Calabi-Yau structures induce a Poisson brackets on representation spaces $({\rm Rep}_n A)^{Gl_n}$ for any associative algebra $A$.

\end{abstract}



\section{Introduction}

We consider the structures introduced in
\cite{Ktalk}, \cite{KV}, \cite{Seidel},
which are cyclically invariant with respect to the natural inner form solutions of the Maurer-Cartan equation on the algebra $A \oplus A^*$, for any graded associative algebra $A$. This structure is called pre-Calabi-Yau algebra. We show, how the double Poisson bracket \cite{VdB}
appear as a particular part of a pre-Calabi-Yau structure.

It was suggested in \cite{Kf} that cyclic structure on $A_{\infty}$-algebra with respect to  non-degenerate inner form should be considered as a symplectic form on the formal noncommutative manifold.
Here we demonstrate that a natural inner form on $A \oplus A^*$, namely pre-Calabi-Yau structure gives rise to a noncommutative version of a Poisson bracket.

Indeed,  we check that pre-Calabi-Yau structure of appropriate kind on $A$ induce a $Gl_n$ invariant Poisson bracket on the representation spaces  $({\rm Rep}_n A)$ of $A$.
 More precisely, we found the way to associate to a pre-Calabi-Yau structure, defined by an $A_\infty$-structure on
 $A\oplus A^*=(A\oplus A^*, m=\sum\limits_{i=2, i\neq 4}^{\infty}m_i^{(1)})$
 a bracket, which satisfies all axioms of the double Poisson algebra \cite{VdB}.
This allows to consider pre-Calabi-Yau structures as a noncommutative version of Poisson structures according to the ideology introduced and developed in \cite{Kf,KR}, saying that noncommutative structure should manifest as a corresponding commutative structure on representation spaces.
We show how double Poisson bracket can be constructed from a pre-CY algebra $A=(A \oplus A^*,  m=\sum\limits_{i=2, i\neq 4}^{\infty}m_i^{(1)})$ with $m_4=0$. This indicates the case when Jacobi identity can be obtained form Maurer-Cartan equation exactly, without any  additional correcting terms. Thus to clarify the picture and to do it for arbitrary algebras we consider this case separately here, in spite the noncmmutative Poisson structure can be obtained from an arbitrary pre-Calabi-Yau structure, without this restriction on $m_4$ (as, for example, we show in \cite{}).

The way we establish the correspondence between the two structures is the following.
First, we associate to a solution of the Maurer-Cartan equation of type B the Poisson algebra structure. Most subtle point here is the choice of the definition  of the bracket via the pre-Calabi-Yau structure.

\begin{theorem}\label{main}
Let we have $A_{\infty}$-structure on $(A\oplus A^*,  m=\sum\limits_{i=2, i\neq 4}^{\infty}m_i^{(1)})$.
Define the bracket by the formula

$$
\langle g\otimes f,\d b,a\c\rangle:=\langle m_3(a,f,b), g \rangle,
$$
where $a,b\in A$, $f,g\in A^*$ and $m_3(a,f,b)=c\in A$ corresponds to the component of type B of the solution to  the Maurer-Cartan, i.e. $m_3$: $A\times A^*\times A\to A$. Then this bracket does satisfy all axioms of the double Poisson algebra.

Moreover, pre-Calabi-Yau structures
 of type $B$ (corresponding to the tensor
$A\otimes A^*\otimes A\otimes A^*$ or $A^*\otimes A\otimes A^*\oo A$) with $m_i=0, i\geq 4$ are in one-to-one correspondence with the double Poisson brackets
$\d\cdot,\cdot\c:A\otimes A\to A\otimes A$ for an  arbitrary associative algebra $A$.

\end{theorem}

Here we concentrate  on the non-graded version of the double Poisson structure. We show how it could be obtained as a part of the solution of the Maurer-Cartan equation
on the algebra $A \oplus A^*$, which is already a graded  object  {otherwise the Maurer-Cartan equation would be trivial). Namely, we consider the grading on $R=A \oplus A^*$, where $R_0=A, \,\, R_1=A^*$, and find it quite amazing how graded continuation of an arbitrary non-graded associative algebra can induce an interesting non-graded structure on $A$ itself.
In order the non-graded version of the double Poisson algebra to be induced on associative algebra $A$ (sitting in degree zero) it have to be 2-pre-Calabi-Yau structure in case of  this grading.
Analogous results  for an arbitrary graded associative algebra will be treated elsewhere.

The contents of this paper is somewhat extended version of the  preprint \cite{IK}.

The structure of the paper is the following.
We explain the notions of strong homotopy associative algebra ($A_\infty$-algebra) and of pre-Calabi-Yau structure on $A_\infty$ or
graded associative algebra in Section~\ref{def}.


In Section~\ref{MC45}
 we suggest a way to define a double Poisson bracket out of the part of the
  solutions to the Maurer-Cartan equation corresponding to exclusively
operations of the type $A\otimes A^*\otimes A\otimes A^*$, that is of type B.
The axioms  (double Jacobi identity)  of double Poisson bracket obtained in such a way from pre-CY structures with $m_4=0$ are proved. This shows that  pre-CY structures with $m_4=0$ induce the double Poisson bracket. Moreover, for the structures of type B with $m_i=0, i\geq 4$ there is a one-to-one correspondence between those structures and Poisson brackets, defined by our formula from Theorem~\ref{main}.


In Section~\ref{rep} we  discuss how  pre-Calabi-Yau structures via double Poisson bracket induce a Poisson
structures on representation spaces of an arbitrary associative algebra.

\section{Finite and infinite dimensional pre-Calabi-Yau algebras}\label{def}

We deal here with the definition of a d-pre-Calabi-Yau structure on $A_{\infty}$-algebra.
Further in the text we consider mainly pre-Calabi-Yau structures on an associative algebra $A$.
Since in the definition of pre-Calabi-Yau structure the main ingredient is  $A_{\infty}$-structure on $A\oplus A^*$ we start with the definition of $A_{\infty}$-algebra, or {\it strong homotopy associative algebra} introduced by Stasheff \cite{St}.

In fact, there are two accepted conventions of grading of an $A_\infty$-algebra. They differ by a shift in numeration of graded components.
In one convention, we call it {\it Conv.1}, each operation has degree 1, while the other is determined by making the binary operation to be of degree 0, and the degree
of operation of arity $n$, $m_n$ to be $2-n$. This second convention will be called
{\it Conv.0}. If the degree of element $x$ in {\it Conv.0} is  ${\rm deg}x=|x|$, then shifted degree in $A^{sh}=A[1]$, which fall into {\it Conv.1}, will be ${\rm deg^{sh}}x=|x|'$, where  $|x|'=|x|-1$, since $x \in A^i=A[1]^{i+1}$.

  The formulae for the graded Lie bracket, Maurer-Cartan equations and cyclic invariance of the inner form are different in different conventions. Since we mainly will use the {\it Conv.1}, but in a way need {\it Conv.0} as well we sometimes present both of them.

Let $A$ be a $\Z$ graded vector space $A=\mathop{\oplus}\limits_{n \in \Z} A_n$. Let $C^l(A,A)$ be Hochschild cochains $C^l(A,A)=\underline{Hom}_{}(A[1]^{\oo l}, A[1])$, for $l\geq 0$, $C^{\bullet}(A,A)= \prod\limits_{k\geq 1} C^l(A,A).$

On
$C^{\bullet}(A,A)[1]$ there is a natural structure of graded pre-Lie algebra, defined via composition:

$$\circ:
C^{l_1}(A,A) \oo C^{l_2}(A,A) \to C^{l_1+l_2-1}(A,A):$$

$$ f \circ g (a_1 \oo ... \oo a_{l_1+l_2-1}) =$$

 $$\sum (-1)^{| g| \sum\limits_{j=1}^{i-1} |a_j|} f(a_1 \oo ... \oo a_{i-1} \oo g(a_1 \oo ... \oo a_{i+l_2+1}) \oo ... \oo a_{l_1+l_2-1})$$

 The operation $ \circ$ defined in this way does satisfy the graded right-symmetric identity:

 $$(f,g,h)=(-1)^{|g||h|}(f,h,g)$$

 where

 $$(f,g,h)=(f\circ g)\circ h - f\circ (g \circ h).$$

 As it was shown in  \cite{G} the graded commutator on a graded pre-Lie algebra defines a graded Lie algebra structure.

 Thus the Gerstenhaber bracket $[-,-]_G$:

$$[f,g]_G= f\circ g - (-1)^{| f|| g|} g \circ f$$

makes $C^{\bullet}(A)$ into a graded Lie algebra. Equipped with the derivation $d={\rm ad} \,\, m_2$, $\,\,\, (C^{\bullet}(A), m_2)$ becomes a DGLA, which is a Hochschild cohomological complex.

 With respect to the Gerstenhaber bracket $[-,-]_G$  we have  the Maurer-Cartan equation

\begin{equation}\label{MC}
[m^{(1)},m^{(1)}]_G=\sum\limits_{p+q=k+1} \sum\limits_{i=1}^{p-1} (-1)^\epsilon m_p(x_1,\dots,x_{i-1},m_q(x_j,\dots,x_{i+q-1}),\dots,x_k)=0,
\end{equation}
where
$$
\epsilon=|x_1|'+{\dots}+|x_{i-1}|',\qquad |x_i|'=|x_i|-1={\rm deg}x_i-1
$$




The  Maurer-Cartan in  {\it Conv.0} is:

\begin{equation}\label{MC0}
[m^{(1)},m^{(1)}]=\sum\limits_{p+q=k+1} \sum\limits_{i=1}^{p-1} (-1)^\epsilon m_p(x_1,\dots,x_{i-1},m_q(x_j,\dots,x_{i+q-1}),\dots,x_k)=0,
\end{equation}
where
$$
\epsilon=i(q+1)+q(|x_1|+{\dots}+|x_{i-1}|,
$$

\begin{definition} An element $m^{(1)} \in C^{\bullet}(A,A)[1]$ which satisfies the Maurer-Cartan equation $[m^{(1)}, m^{(1)}]_G$ with respect to the Gerstenhaber bracket $[-,-]_G$ is called an
$A_{\infty}$-structure on $A$.

\end{definition}

Equivalently, it can be formulated in a more compact way as a coderivation on the coalgebra of the bar complex of $A$.








In particular, associative algebra with zero derivation $A=(A, m=m_2^{(1)})$ is an example of $A_{\infty}$-algebra. The  component of the Maurer-Cartan equation of arity 3, $MC_3$ will say that the binary operation of this structure, the multiplication $m_2$ is associative:

$$(ab)c-a(bc)=dm_3(a,b,c)+(-1)^{\sigma}m_3(da,b,c)+(-1)^{\sigma}m_3(a,db,c)+(-1)^{\sigma}m_3(a,b,dc)$$

We can give now definition of pre-Calabi-Yau structure (in $Conv.1$).

\begin{definition}\label{pcy1} A d-pre-Calabi-Yau structure on a finite dimensional $A_{\infty}$-algebra $A$ is

 (I). an $A_{\infty}$-structure on $A \oplus A^*[1-d]$,

 (II). cyclic invariant with respect to natural non-degenerate pairing on $A \oplus A^*[1-d]$,
meaning:

$$\l m_n(\a_1,...,\a_n), \a_{n+1}\r=(-1)^{|\a_1|'(|\a_2|'+...+|\a_{n+1}|')}\l m_n(\a_2,...\a_{n+1}),\a_1)\r$$
where the inner form $\l,\r$ on $A+A^*$   is defined naturally as
$\l (a,f),(b,g)\r=f(b)+(-1)^{|g|' |a|'} g(a)$ for $a,b \in A, f,g \in A^*$

(III)
and such that $A$ is $A_{\infty}$-subalgebra in $A \oplus A^*[1-d]$.
\end{definition}


The signs in this definition written in {\it Conv.1}, are
assigned according to the Koszul rule. It is not quite the case for {\it Conv.0}, where
the
cyclic invariance with respect to the natural non-degenerate pairing on $A \oplus A^*[1-d]$, from (II) sounds:

$$\l m_n(\a_1,...,\a_n), \a_{n+1}\r=(-1)^{n+|\a_1|'(|\a_2|'+...+|\a_{n+1}|')}\l m_n(\a_2,...\a_{n+1}),\a_1\r$$

The appearance of the arity $n$, which influence the sign in this formula, does not really fit with the Koszul rule, this is the feature of the {\it Conv.0}, and this is why it is more convenient to work with the {\it Conv.1}.

As we will need to refer to these later, let us define separately the cyclic invariance condition and inner form symmetricity in {\it Conv.1}:

\begin{equation}\label{cyc1}
\l m_n(\a_1,...,\a_n), \a_{n+1}\r=(-1)^{|\a_1|'(|\a_2|'+...+|\a_{n+1}|')}\l m_n(\a_2,...\a_{n+1}),\a_1)\r
\end{equation}

\begin{equation}\label{sym2}
\langle x,y \rangle = - (-1)^{|x|'\,|y|'} \langle y,x \rangle
\end{equation}

The notion of pre-Calabi-Yau algebra introduced in
\cite{KV}, \cite{Seidel}
use the fact that $A$ is finite dimensional, since there is no natural grading on the dual algebra $A^*={\rm Hom}(A,\K)$, induced form the grading on $A$
 in  infinite dimensional case.
The general definition suitable for infinite dimensional algebra was given in \cite{KV}, \cite{Ktalk}, and it is equivalent to the  definition, where the ${\rm Hom}(A,\K)$ is substituted with the graded version: $A^*=\oplus (A_n)^*=\underline{{\rm Hom}}(A,\K)$, if graded components are finite dimensional.
We will give this general definition and show the equivalence further in this section.



{\bf Example.}  The most simple example of pre-Calabi-Yau structure demonstrates that this structure does exist on any associative algebra. Namely, the structure of associative algebra on $A$ can be extended to the associative structure on $A\oplus A^*[1-d]$ in such a way, that the natural inner form is (graded)cyclic with respect to this multiplication. This amounts to the following fact: for any $A$-bimodule $M$ the associative multiplication on $A \oplus M$ is given  by
$ (a+f)(b+g)=ab+af+gb.$ In this simplest situation both structures on $A$ and on $A+A^*$ are in fact associative algebras.
More examples one can find in \cite{I},
 \cite{Rub}, \cite{DK}.

One can reformulate the above definition without $A^*$, using the inner product, to change inputs and outputs of operations, and by this to substitute $A^*$ with $A$, as it was done in \cite{KV}. The $A_{\infty}$-structure on $A \oplus A^*$ means first of all the bunch of linear maps
$$m_N: (A \oplus A^*)^N \to A \oplus A^*.$$
Such a map splits as a collection of linear maps of the type

$$\xi= m_{p_1,...,p_l}^{q_1,...,q_l}: A^{\oo p_1} \oo A^{*\oo q_1} \oo ... \oo A^{\oo p_l} \oo A^{*\ootimes q_l} \to A\, ({\rm or}\, A^*)$$
where $\sum p_i + q_i=N$, $0 \leq p_i \leq N$.

These could be interpreted,  using the inner product, as  tensors of the type,
$ A^{\oo p_1} \oo A^{*\oo q_1} \oo ... \oo A^{\oo p_l} \oo A^{*\ootimes q_l}, $
$\sum p_i+q_i=N+1,$
and graphically depicted as operations where incoming edges correspond to  elements of $A$,  outgoing edges to elements of $A^*$ and the marked point correspond to  the output of operation $m_{p_1,...,p_l}^{q_1,...,q_l}$. (Of course, due to cyclic invariance operations with different marked points are equal up to a sign, but to keep track of the signs, marked point is needed).
This gives rise to the definition below.

First, we shell define
{\it higher Hochschild cochains}  and {\it generalised necklace bracket}.

\begin{definition}\label{hh} For $k\geq 1$ the space of {\it k-higher Hochschild cochains} is defined as

$$
C^{(k)}(A):= \prod_{r \geq 0} \underline{Hom}(A[1]^{\oo r}, A[1]^{\oo k})
$$

$$
=\prod_{r_1,...,r_k \geq 0} \underline{Hom}(\mathop{\otimes}\limits_{i=1}^k A[1]^{\oo r_i}, A[1]^{\oo k})
$$

Denote by $C^{(\bullet)}(A)= \prod_{k \geq 1}  C^{(k)}(A)$ the space of all higher Hochschild cochains.
\end{definition}

Note, that  $C^{(1)}(A)= C^{\bullet}(A,A)$ is the space of usual Hochschild cochains.

We can see that element of the higher Hochschild cochain
can be interpreted as
operation with $r$ incoming  and $k$ outgoing edges. There is a marked point in the picture as well, and due to  the cyclic invariance condition one can move this marked point with the change of a sign. Indeed, suppose that the last position is marked, then it can be moved to the one but last using the formula:

 $$(-1)^{|x_n|(|x_1|+...+|x_{n-1}|)}<x_1,m(x_2,...,x_n)>= <x_n,m(x_1,...,x_{n-1})>$$
where $x_i \in A \,\, {\rm or} \,\, A^*.$
Thus, just higher  Hochschild cochains, without a specified point will appear in the definition of pre-Calabi-Yau structure.

 The composition of two operations of this kind translates
  according to definition~\ref{pcy1} to the explained above picture  via the notion of generalised necklace bracket:

 \begin{definition}\label{neckl} The {\it generalised necklace bracket} between two elements $f,g \in C^{(k)}(A)$ is given as
  $[f,g]_{gen.neckl}= f\circ g - (-1)^{\sigma} g \circ f,$ where composition $f\circ g$ consists of inserting all outputs of $g$ to all inputs from $f$ with signs assigned according to the Koszul rule.
  \end{definition}

Again, since the defined above composition $f\circ g$ makes $C^{(\bullet)}$ into a graded pre-Lie algebra, the generalised necklace bracket obtained from it as a graded commutator, makes $C^{(\bullet)}$ into a graded Lie algebra.

\begin{definition}\label{pcy2}
Let $A$ be  a $\Z$-graded space $A=\oplus A_n$. The pre-Calabi-Yau structure on $A$ is an element from the space of higher Hochschild cochains $C^{(\bullet)}$, $m=\sum_{k\geq 0} m^{(k)}$, $m^{(k)} \in C^{(k)}(A)$, which is a  solution to the Maurer-Cartan equation $[m,m]_{gen.neckl}=0$ with respect to generalised necklace bracket.

\end{definition}

Any such solution makes $C^{(\bullet)}(A)$ into a DGLA with the differential ${\rm ad} \, m$.

\begin{definition}\label{pcy3}
The pre-Calabi-Yau structure on a $\Z$-graded space $A=\oplus A_n$ is a cyclically invariant $A_{\infty}$-structure on $A \oplus A^*[1-d]$, where $A^*$ is understood as  $A^*=\oplus (A_n)^*=\underline{{\rm Hom}}(A,\K)$.
\end{definition}

\begin{proposition}
The definitions of pre-Calabi-Yau structures \ref{pcy2} and \ref{pcy3}  are equivalent, when dim$A_n<\infty\,\, \forall n$.
\end{proposition}





\begin{proof}
To demonstrate this we will start with an element $m=\sum_{k\geq 0} m^{(k)}$, $m^{(k)} \in C^{(k)}(A)$,  $m^{(k)}=A^{\oo r} \to A^{\oo k} $ depicted as an operation with r incoming arrows, k outgoing arrows, and one marked point.

From this data we construct  a collection
of operations $m_n:(A\oplus A^*)^{\otimes n}\to A\oplus A^*$,  to form an $A_{\infty}$-structure on $A \oplus A^*$.



So let us have an element $\xi$ in the tensor product (in some order) of $r$ copies of $A$ and $k$ copies of $A^*$,
where the last position is specified.
Thus we have an operation $E:A^{\otimes r}\to A^{\otimes k}$
with one fixed entry. This defines an element $\widehat\xi\in (A^*)^{\otimes r}\otimes A^{\otimes k}$ (by means of
the natural pairing) such that

$$
\langle E(a_1\oo...\oo a_r), f_1\oo...\oo f_k \rangle = \langle \widehat\xi, a_1\oo...\oo a_r \oo f_1\oo...\oo f_k \rangle
$$
Note that here we use the equality $A^{**}=A$, which is true only for finite dimensional spaces.  We should make sure
that we use duals satisfying $A^{**}=A$, as it is done in definition~\ref{pcy3}, when graded components are finite dimensional.

Now we can define an operation from the $A_\infty$-structure on $A\oplus A^*$ corresponding to the above operation $E$,
$$
m_{n-1}(a_1,f_1,\dots,\widehat{f_k})
$$
if the marked point have an outgoing edge and
$$
m_{n-1}(a_1,f_1,\dots,\widehat{a_r})
$$
if the marked point has an incoming edge. Here $n=k+r$ and the order of entries of elements from $A$ and from $A^*$ is dictated by the order in $\xi$.
In these two cases we define $m_{n-1}$ as follows:
\begin{align*}
&\langle f_k,m_{n-1}(a_1,f_1,\dots,\widehat{f_k})\rangle =\langle \widehat\xi,a_1\otimes\dots\otimes a_r\otimes f_1\otimes\dots\otimes f_k\rangle;
\\
&\langle a_r,m_{n-1}(a_1,f_1,\dots,\widehat{a_r})\rangle =\langle \widehat\xi,a_1\otimes\dots\otimes a_r\otimes f_1\otimes\dots\otimes f_k\rangle.
\end{align*}

\end{proof}

In spite  definition \ref{pcy2} looks more beautiful and reveals nice graphically presented  connection with A-infinity structure,  we will use definition \ref{pcy3},
since we find it easier to work with and make sure all details are correct.

\section{Structure of the Maurer-Cartan equations }\label{str}

The general Maurer-Cartan equations on $C=A\oplus A^*$ for the operations $m_n:C[1]^n\to C[1]$ have the shape
$$
\sum_{p+q=k+1} \sum\limits_{i=1}^{p-1} (-1)^\epsilon m_p(x_1,\dots,x_{i-1},m_q(x_j,\dots,x_{i+q-1}),\dots,x_k),
$$
where
$$
\epsilon=|x_1|'+{\dots}+|x_{i-1}|',\qquad |x_s|'={\rm deg}x_s-1
$$


The equations we get from the Maurer-Cartan in arities  four and five, which are relevant for comparing with the Leibniz and Jacibi identities for the double bracket, will look as follows.

In arity 4, Maurer-Cartan equation, $MC_4$ reads:

$$  m_3(x_1x_2, x_3, x_4)+ (-1)^{|x_1|'} m_3(x_1,x_2x_3, x_4)+ (-1)^{|x_1|'+|x_2|'}m_3(x_1,x_2,x_3 x_4)+ $$
$$(-1)^{|x_1|'}m_2(x_1, m_3(x_2,x_3, x_4))+ m_2(m_3(x_1,x_2,x_3), x_4)=0$$



In arity 5, Maurer-Cartan equation, $MC_4$ reads:

$$
 m_3(m_3(x_1, x_2,x_3), x_4, x_5) + (-1)^{|x_1|'} m_3(x_1, m_3( x_2,x_3, x_4), x_5)
+ $$
$$(-1)^{|x_1|'+|x_2|'} m_3(x_1, x_2, m_3( x_3, x_4, x_5) )=0
$$

Operations of arity 4  are absent due to our condition  that $m_4=0$
in $A_{\infty}$-structure on $A \oplus A^*$.

Since we have Maurer-Cartan equations on $A\oplus A^*$, it essentially means that any equation splits into the set of equations
with various distributions of inputs/outputs from $A$ and $A^*$.
Note that solutions of the Maurer-Cartan which are interesting for us correspond to operations $A \oo A \to A\oo A$ (which can serve as a double bracket). These are operations from tensors with exactly two $A$th and two $A^*$th.

Remind that an operation, say, $A^*\times A^*\times A\to A^*$ can be naturally interpreted as an element of the space $A\otimes A\otimes A^*\otimes A^*$
and this tensor due to cyclic invariance of the structure equals to its cyclic permutations up to sign, in this case $A^*\otimes A\otimes A\otimes A^*$,
$A^*\otimes A^*\otimes A\otimes A$ and $A\otimes A^*\otimes A^*\otimes A$. There is another type of tensor from $A\otimes A^*\otimes A\otimes A^*$ for which
there is only one cyclic permutation $A^*\otimes A\otimes A^*\otimes A$.
Due to cyclic invariance
$$
\langle m_3(f,a,g),b\rangle = \pm\langle m_3(b,f,a),g\rangle
$$
operation  $A^*\times A \times A^*\to A^*$ corresponding to tensor $A\otimes A^*\otimes A\otimes A^*$ is the same as operation  $A\times A^*\times A \to A$ corresponding to tensor $A^*\otimes A \otimes A^*\otimes A.$ These tensors encode the second type of operations.

Two types of operations mentioned above which are different up to cyclic permutation on tensors will serve as variables in the equations we obtain from the Maurer-Cartan.

Let us list 6 tensors corresponding to 2 types of operations, of which we will think as of  two types of  {\it main variables} in MC equations.

$$
\begin{array}{llll}
{\rm Type \,\, A}&&{\rm Type \,\, B}&\\
A^*\otimes A\otimes A\otimes A^*,&A\times A^*\times A^*\to A^*,&&\\
A^*\otimes A^*\otimes A\otimes A,&A\times A\times A^*\to A,&A\otimes A^*\otimes A\otimes A^*,&A^*\times A\times A^*\to A^*,\\
A\otimes A^*\otimes A^*\otimes A,&A^*\times A\times A\to A,&A^*\otimes A\otimes A^*\otimes A,&A\times A^*\times A\to A,\\
A\otimes A\otimes A^*\otimes A^*,&A^*\times A^*\times A\to A^*.  &&
\end{array}
$$

\begin{definition} We say that operations corresponding to the tensor $A\otimes A\otimes A^*\otimes A^*$ (and its cyclic permutations) are operations of {\it type A}, and operations corresponding to the tensor $A\otimes A^*\otimes A\otimes A^*$ (and its cyclic permutations) are operations of {\it type B}.
\end{definition}

These two variables, being cyclicly invariant tensors can be depicted as follows:

Graphically these cyclically invariant operations could be depicted  as follows for type A and B respectively.

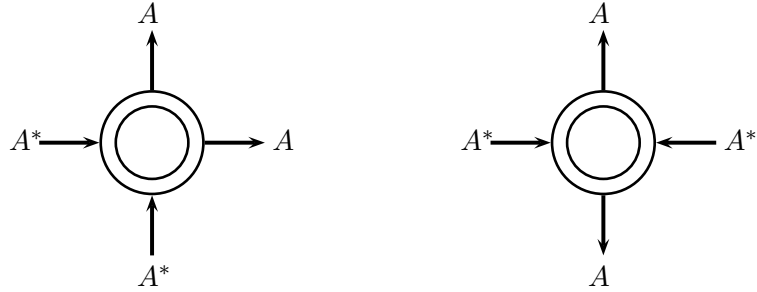
\def\vvv{\vrule width0pt height1.2cm depth 1.2cm}
\begin{pspicture}(0,0)(13,4)
\pscircle[linewidth=1pt](5,2){0.7}
\pscircle[linewidth=1pt](5,2){0.5}
\psline[linewidth=1.5pt]{->}(5,0.5)(5,1.3)
\psline[linewidth=1.5pt]{->}(5,2.7)(5,3.5)
\psline[linewidth=1.5pt]{->}(3.5,2)(4.3,2)
\psline[linewidth=1.5pt]{<-}(6.5,2)(5.7,2)
\put(6.6,1.9){$A$}
\put(3.1,1.9){$A^*$}
\put(4.8,3.6){$A$}
\put(4.8,0.1){$A^*$}
\pscircle[linewidth=1pt](11,2){0.7}
\pscircle[linewidth=1pt](11,2){0.5}
\psline[linewidth=1.5pt]{<-}(11,0.5)(11,1.3)
\psline[linewidth=1.5pt]{->}(11,2.7)(11,3.5)
\psline[linewidth=1.5pt]{->}(9.5,2)(10.3,2)
\psline[linewidth=1.5pt]{->}(12.5,2)(11.7,2)
\put(12.6,1.9){$A^*$}
\put(9.1,1.9){$A^*$}
\put(10.8,3.6){$A$}
\put(10.8,0.1){$A$}
\end{pspicture}



But we will mainly use for calculations the above row notations, since they are more suitable for following the signs, which are crucially important in some of our calculations, for example, for the result on one-to-one correspondence between part of pre-CY structure and a double Poisson structure.



The other operations which are also variables in the Maurer-Cartan equation, correspond to the tensors of length four, containing not exactly two $A$ and two $A^*$. We call them
{\it secondary type}
variables, as opposed to the main type, consisting of variables of type $A$ and $B$.
 So, {\it secondary type}
variables correspond to the cyclic invariant tensors with one $A^*$:
$A\otimes A\otimes A\otimes A^*$, we call this $C1$  or with one $A$: $A^*\otimes A^*\otimes A^*\otimes A$, this we call $C2$ (there are eight corresponding operations), as well as two operations corresponding to each of  tensors $A\otimes A\otimes A\otimes A$ and $A^*\otimes A^*\otimes A^*\otimes A^*$, which are variables $C3$ and $C4$ respectively.

Graphically these cyclically invariant operations could be depicted  as follows.

\vskip0.5cm

$C_1, C_2:$

\vskip0.5cm

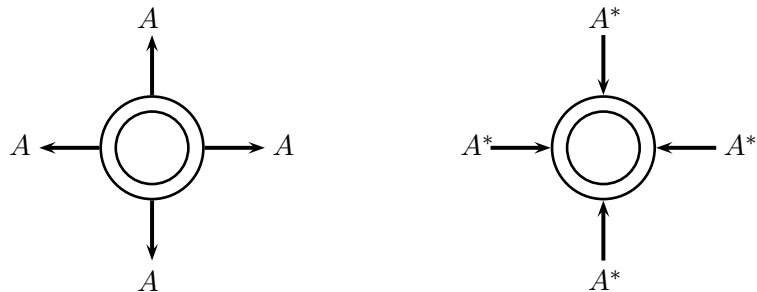
\def\vvv{\vrule width0pt height1.2cm depth 1.2cm}
\begin{pspicture}(0,0)(13,4)
\pscircle[linewidth=1pt](5,2){0.7}
\pscircle[linewidth=1pt](5,2){0.5}
\psline[linewidth=1.5pt]{<-}(5,0.5)(5,1.3)
\psline[linewidth=1.5pt]{->}(5,2.7)(5,3.5)
\psline[linewidth=1.5pt]{<-}(3.5,2)(4.3,2)
\psline[linewidth=1.5pt]{<-}(6.5,2)(5.7,2)
\put(6.6,1.9){$A$}
\put(3.1,1.9){$A$}
\put(4.8,3.6){$A$}
\put(4.8,0.1){$A$}
\pscircle[linewidth=1pt](11,2){0.7}
\pscircle[linewidth=1pt](11,2){0.5}
\psline[linewidth=1.5pt]{->}(11,0.5)(11,1.3)
\psline[linewidth=1.5pt]{<-}(11,2.7)(11,3.5)
\psline[linewidth=1.5pt]{->}(9.5,2)(10.3,2)
\psline[linewidth=1.5pt]{->}(12.5,2)(11.7,2)
\put(12.6,1.9){$A^*$}
\put(9.1,1.9){$A^*$}
\put(10.8,3.6){$A^*$}
\put(10.8,0.1){$A^*$}
\end{pspicture}

\vskip0.5cm

$C_3, C_4:$

\vskip0.5cm

\def\vvv{\vrule width0pt height1.2cm depth 1.2cm}
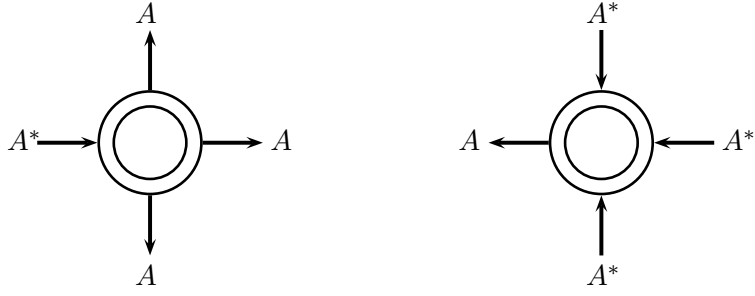
\begin{pspicture}(0,0)(13,4)
\pscircle[linewidth=1pt](5,2){0.7}
\pscircle[linewidth=1pt](5,2){0.5}
\psline[linewidth=1.5pt]{<-}(5,0.5)(5,1.3)
\psline[linewidth=1.5pt]{->}(5,2.7)(5,3.5)
\psline[linewidth=1.5pt]{->}(3.5,2)(4.3,2)
\psline[linewidth=1.5pt]{<-}(6.5,2)(5.7,2)
\put(6.6,1.9){$A$}
\put(3.1,1.9){$A^*$}
\put(4.8,3.6){$A$}
\put(4.8,0.1){$A$}
\pscircle[linewidth=1pt](11,2){0.7}
\pscircle[linewidth=1pt](11,2){0.5}
\psline[linewidth=1.5pt]{->}(11,0.5)(11,1.3)
\psline[linewidth=1.5pt]{<-}(11,2.7)(11,3.5)
\psline[linewidth=1.5pt]{<-}(9.5,2)(10.3,2)
\psline[linewidth=1.5pt]{->}(12.5,2)(11.7,2)
\put(12.6,1.9){$A^*$}
\put(9.1,1.9){$A$}
\put(10.8,3.6){$A^*$}
\put(10.8,0.1){$A^*$}
\end{pspicture}

\vskip0.5cm

Again, for our calculations it will be more convenient to present them as operations, written in a row (e.i. with the fixed starting point), so in these denotations we have the following operations corresponding to variables  $C_1, C_2, C_3, C_4$:


$$
\begin{array}{llll}
{\rm Secondary \,\,\, type}&&&\\
{\rm variable \,\,\, C3}&&{\rm variable \,\,\, C4}&\\
A\otimes A\otimes A\otimes A^*,&A^*\times A^*\times A^*\to A^*,&A^*\otimes A^*\otimes A^*\otimes A,&A\times A\times A\to A,\\
A^*\otimes A\otimes A\otimes A,&A\times A^*\times A^*\to A,&A\otimes A^*\otimes A^*\otimes A^*,&A^*\times A\times A\to A^*,\\
A\otimes A^*\otimes A\otimes A,&A^*\times A\times A^*\to A,&A^*\otimes A\otimes A^*\otimes A^*,&A\times A^*\times A^*\to A^*,\\
A\otimes A\otimes A^*\otimes A,&A^*\times A^*\times A\to A,&A^*\otimes A^*\otimes A\otimes A^*,&A\times A\times A^*\to A^*,\\
{\rm variable \,\,\, C1}&&{\rm variable \,\,\, C2}&\\
A\otimes A\otimes A\otimes A,&A^*\times A^*\times A^*\to A,&A^*\otimes A^*\otimes
A^*\otimes A^*,&A\times A\times A\to A^*.
\end{array}
$$

Let us look at what we can get from the Maurer-Cartan in arity 5.

First consider the input row containing 4 or more entries from $A$ (or $A^*$). It is easy to check that in this case all terms of equations we get contain
secondary type variables. For example, consider the input $A,A,A^*,A,A$. The term $m_3(m_3(a,f,b),c,d)$ is zero if $m_3(a,b,f)\in A$, since $A$ is associative
algebra and $m_3(a_1,a_2,a_3)=0$ for all $a_1,a_2,a_3\in A$. If $m_3(a,b,f)\in A^*$ then the operation is of secondary type, from tensor
$A^*\otimes A^*\otimes A\otimes A^*$.

Another group of equations correspond to input containing three $A$ (or $A^*$). These are divided according to what is the output of the corresponding
operation of arity 5. In case of operations with 3 inputs from $A$, two inputs from $A^*$ and output from $A^*$ as well as 3 inputs from $A^*$,
two inputs from $A$ and output from $A$, all terms of the equations still contain at least one variable of secondary type.

This property of equations will allow us to restrict any solution of the Maurer-Cartan to the ones containing only main variables (take the projection of solution to
the space of main variables, and ensure that we have a solution again).

In the cases of operations with 3 inputs from $A$, two inputs from $A^*$ and output from $A$ as well as 3 inputs from $A^*$,
two inputs from $A$ and output from $A^*$, all terms of the equations  contain only variables of the main type. Each of these cases corresponds
to 10 (5 choose 2) equations on main variables. We consider their structure in more detail. These equations on main variables contain both variables of
types $A$ and of $B$. Call variables of type $B$ by $X$'s and of type $A$ by $Y$'s. Then the system of equations again splits into those, each term
of which contains a $Y$ variable and those which are equations only on $X$'s.

\begin{lemma}
Any equation on main variables coming  from $MC_5$ either containing  only terms $XX$, i.e. only variables of type $B$ or each term contains at least one variable $Y$ - variable of type $A$.
\end{lemma}

\begin{proof}
Let us see from which  inputs terms of type $XX$ can appear. There are two operations of type $B$: I. $A\times A^* \times A \to A$ and II. $ A^* \times A \times A^* \to A^*$.
Consider the case of composition of the type $m_3(x_1, m_3(x_2,x_3,x_4),x_5))$. In case I. to have a composition of two operations of type B we forced to start with input row
$A^*(AA^*A)A^*$. In case II. to have a composition of two operations of type B we forced to start with input row $A(A^*AA^*)A$. The remaining types of compositions: $m_3( m_3(x_1,x_2,x_3), x_4, x_5))$ and  $m_3(x_1, x_2, m_3(x_3,x_4,x_5))$  analogously give the same result. Thus the only rows of inputs from which $XX$ term can appear are those two rows. We see moreover that no compositions containing variable Y (operation of type A) appear from this row of input. Thus variables $X$ and $Y$ are separated in the above sense in this system of equations.
\end{proof}


This structure of the system of equations on operations which constitute an unwrapped Maurer-Cartan equation will be a key to relate any pre-Calabi-Yau structure
concentrated in appropriate arities to the double Poisson bracket. Each equation that we get from $MC_5$ consists of 'quadratic' terms, meaning terms involving
two operations. This system of equations has the feature that in no equation both terms containing two $X$ variables and $XY$ or $YY$ terms appear. These
terms are separated. The above arguments  allow us to see that

\begin{proposition}\label{121121}
Projection of any $MC_5$ solution to the $B$-type component is also a solution of  $MC_5$.
\end{proposition}

\section{Solutions to the Maurer-Cartan equations in arity four and five and double Poisson bracket}\label{MC45}

In this section we show that the pre-Calabi-Yau structures of type B, namely the ones which are solutions of type B (corresponding to the tensor $A\oo A^*\oo A\oo A^*$ or $ A^*\oo A\oo A^* \oo A$) of the Maurer-Cartan equation on $A\oplus A^*$,  are in
one-to-one correspondence with the non-graded double Poisson brackets.

 We choose the main example of grading on $A+A^*$ in order to get correspondence with the non-graded double Poisson bracket. Namely, in order to have multiplication on $A$ to be of degree 0 (as it should be in {\it Conv.0} ), we have to   have $A_0= A$. Then, in order for the type $B$ operations (the most interesting part of the solution of  the Maurer-Cartan, which is a ternary operation) to make sense, i.e. according to the {\it Conv.0}, to be of degree $-1$, we need $A^*$ to be in
the component of degree $1$. That is, $R=A\oplus A^*$ is graded by $R_0=A$ and $R_1=A^*$.

 Now we shift this grading by one, to use more convenient formulae of {\it Conv.1}. Thus we get in $A^{sh}=A[1]$ $A^{sh}_{-1}=A$, and $R^{sh}=A^{sh}+A^{*sh}$ is graded by
 $R^{sh}_{-1}=A$,  $R^{sh}_{0}=A^*$, that is $A$ will have degree $-1$, and $A^*$, degree $0$, when we are in shifted situation, and in {\it Conv.1}.

Let $A$ be an arbitrary associative algebra $A=(A,m=m_2^{(1)})$ with a pre-Calabi-Yau structure given as a cyclicly symmetric
$A_\infty$-structure on $A\oplus A^*$: $(A\oplus A^*,m=m_2^{(1)}+m_3^{(1)})$. We define the double Poisson bracket
via the pre-Calabi-Yau structure, more precisely its component corresponding to the tensor $A\otimes A^*\otimes A\otimes A^*$,
as follows.

\begin{definition}\label{br} The double bracket is  defined as:
$$
\langle g\otimes f,\d b,a\c \rangle:=\langle m_3(a,f,b), g \rangle,
$$

where $a,b\in A$, $f,g\in A^*$ and $m_3(a,f,b)=c\in A$ corresponds to the component of $m_3$:
$A\times A^*\times A\to A$.
\end{definition}

By choosing this definition we set up a one-to-one correspondence between pre-Calabi-Yau structures of type B and double Poisson brackets  from \cite{VdB}. This choice have been done in such a way that it would be possible, having the Maurer-Cartan equation to show, that the double bracket defined above indeed satisfies all axioms of double Poisson bracket. Moreover,  no other identities follows in case $m_i=0, i\geq 4$.
Note, that it is most subtle point, since there are many possibilities for this choice.

We will check that double bracket defined in this way satisfies all axioms of the double Poisson bracket.

Anti-symmetry:

\begin{equation}\label{sym}
\d a,b \c = - \d b,a \c^{op}
\end{equation}
Here $\d b,a \c^{op}$ means the twist in the tensor product, i.e. if $\d b,a \c =
\sum\limits_i b_i \oo c_i$, then $\d b,a \c^{op}=\sum\limits_i c_i \oo b_i.$

Double Leibniz:

\begin{equation}\label{Leib}
\d a,bc \c = b \d a,c \c + \d a,b \c c
\end{equation}

and double Jacobi identity:

\begin{equation}\label{Jac}
\d a, \d b,c \c\c_L + \tau _{(123)} \d b \d c,a \c \c_L + \tau _{(132)} \d c \d a,b \c \c_L
\end{equation}
Here for $a \in A\oo A \oo A, \,\,$ and $ \sigma \in S_3\,\,$

$$\tau _{\sigma}(a)=
a_{\sigma^{-1}} (1) \oo a_{\sigma^{-1}} (2) \oo a_{\sigma^{-1}} (3).$$

The $\d\,\c_L$ defined as
$$ \d b, a_1\oo a_n \c_L = \d b, a_1 \c_L \oo a_1\oo ... \oo a_n$$

\begin{theorem}\label{main}
Let we have $A_{\infty}$-structure on $(A\oplus A^*,  m=\sum\limits_{i=2, i\neq 4}^{\infty}m_i^{(1)})$.
Define the bracket by the formula

$$
\langle g\otimes f,\d b,a\c\rangle:=\langle m_3(a,f,b), g \rangle,
$$
where $a,b\in A$, $f,g\in A^*$ and $m_3(a,f,b)=c\in A$ corresponds to the component of type B of the solution to  the Maurer-Cartan, i.e. $m_3$: $A\times A^*\times A\to A$. Then this bracket does satisfy all axioms of the double Poisson algebra.

Moreover, pre-Calabi-Yau structures
 of type $B$ (corresponding to the tensor
$A\otimes A^*\otimes A\otimes A^*$ or $A^*\otimes A\otimes A^*\oo A$) with $m_i=0, i\geq 4$ are in one-to-one correspondence with the double Poisson brackets
$\d\cdot,\cdot\c:A\otimes A\to A\otimes A$ for an  arbitrary associative algebra $A$.

\end{theorem}

\begin{proof}
 Anti-symmetry of the double Poisson bracket reads in these notations:
$$
\langle f\otimes g,\d b,a\c\rangle=-\langle g\otimes f,\d a,b\c\rangle,
$$

So we need to check that
$$
\langle m_3(b,g,a), f\rangle= - \langle m_3(a,f,b), g \rangle.
$$
Indeed, using cyclic invariance, we have

$$\langle f\otimes g,\d b,a\c\rangle=\langle m_3(a,f,b), g\rangle$$
$$=(-1)^{|b|'(|g|'+|a|'+|f|')} \langle m_3(g,a,f), b\rangle =$$
$$=(-1)^{|b|'(|g|'+|a|'+|f|')} (-1)^{|g|'(|a|'+|f|'+|b|')} \langle m_3(a,f,b), g\rangle $$
$$= - \langle m_3(a,f,b), g\rangle = -\langle g\otimes f,\d a,b\c\rangle$$

We used the fact that in our grading $\forall f\in A^*, |f|'=0$ and $\forall a\in A, |a|'=-1.$

By this the anti-symmetry of obtained in this way from pre-Calabi-Yau structure (non-graded) bracket is proven.



Now we deduce the Leibnitz identity from the part of the arity 4 of the Maurer--Cartan equations with inputs from $A$,$A$, $A^*$ and $A$.

General  Maurer--Cartan in arity 4 reads:

$$ m_3(x_1x_2, x_3, x_4)+ (-1)^{|x_1|'}m_3(x_1,x_2x_3, x_4)+ (-1)^{|x_1|'+|x_2|'}m_3(x_1,x_2,x_3 x_4)+
$$
 $$(-1)^{|x_1|'}x_1 m_3(x_2,x_3, x_4)+m_3(x_1,x_2,x_3) x_4=0$$

Applying this to the input $a,b,f,c$  from $A$,$A$, $A^*$, $A$ we have
$$
m_3(ab,f,c)+ (-1)^{|a|'} m_3(a,bf,c) +(-1)^{|a|'+|b|'} m_3(a,b, fc) +(-1)^{|a|'} am_3(b,f,c) + m_3(a,b,f)c=0.
$$



Since we consider solutions containing only $B$-type components, two terms in the equation ($m_3(a,b,f)c$ and $ m_3(a,b, fc)$) vanish, leaving us with

\begin{equation}\label{MC4gr}
m_3(ab,f,c) -  m_3(a,bf,c)  -  am_3(b,f,c) =0.
\end{equation}
after we applied our grading, where $|a|'=-1$ for all $a\in A.$

Now we pair the above equality obtained from $MC_4$ with $g$ (the equality holds if and only if it holds for any pairing with an arbitrary  $g\in A^*$):

$$
\langle m_3(ab,f,c), g \rangle - \langle m_3(a,bf,c), g \rangle - \langle am_3(b,f,c), g \rangle =0.
$$
and express the three terms appearing there via the double bracket.

For doing this we need the following lemma.

\begin{lemma}\label{LR}
The following equalities hold:

$${\rm (R)} \quad \quad \langle  g  \otimes a f, \d b, c\c \rangle= \langle  g  \otimes  f, \d b, c\c a \rangle
$$

$${\rm (L)} \quad \quad \langle  g a \otimes  f, \d b, c\c \rangle= -\langle  g  \otimes  f, a \d b, c\c  \rangle
$$

\end{lemma}

\begin{proof}

We use here Sweedler notations: $\d b, c\c = \sum b_i \otimes c_i  = \sum b' \otimes c''$.

(R)
 $$\langle  g  \otimes a f, \d b, c\c \rangle=
 \sum\langle  g  \otimes  af,  b'  \otimes  c''\rangle=
$$
$$
 \sum\langle  g, b'\rangle \langle  af, c''\rangle \mathop{=}^{(\ref{cycR})}
 \sum\langle  g,  b'\rangle \langle  f,  c'' a\rangle=
 $$
$$
 \langle  g  \otimes  f, \sum  b'  \otimes  c'' a\rangle=
 \langle  g  \otimes f, \d b, c\c a\rangle
 $$

 We use here:
 $$\langle a f, c'' \rangle \mathop{=}^{(\ref{cyc1})} (-1)^{|a|'(|f|'+|c''|')} \langle  f c'', a \rangle \mathop{=}^{(\ref{cyc1})}
 $$
 $$
 (-1)^{|a|'(|f|'+|c''|')} (-1)^{|f|'(|c''|'+|a|')} \langle  c'' a, f \rangle
 $$

$$
 \mathop{=}^{(\ref{sym2})}(-1)^{|a|'(|f|'+|c''|')} (-1)^{|f|'(|c''|'+|a|')}\cdot - (-1)^{|c'' a|'|f|'} \langle  f,c'' a \rangle
 $$

and in our grading, where for all $a\in A, f\in A^*$, $|a|'=-1, |f|'=0$, we get

\begin{equation}\label{cycR}
\langle a f, c'' \rangle = \langle f, c'' a \rangle
\end{equation}
for all $a,c'' \in A, \,\, f\in A^*$

(L)
 $$\langle  g a \otimes  f, \d b, c\c \rangle=
 \sum\langle  g a \otimes  f\rangle \langle  b'  \otimes  c''\rangle=
$$
$$
 \sum\langle  g a \otimes  b'\rangle \langle  f  \otimes  c''\rangle \mathop{=}^{(\ref{cycL})}
 \sum\langle  a b', g \rangle \langle  f,  c'' \rangle
 $$
$$
\mathop{=}^{(\ref{symL})} - \sum \langle  g,  a b'\rangle \langle f, c'' \rangle=
 - \sum \langle g\otimes  f, ab' \oo c'' \rangle=
 -\langle  g  \otimes f, a \d b, c\c \rangle
 $$

 We use here:

 $$\langle ga, b' \rangle \mathop{=}^{(\ref{cyc1})} (-1)^{|g|'(|a|'+|b'|')}
 \langle  a b', g \rangle
 $$

$$
 \langle  a, b' g \rangle \mathop{=}^{(\ref{sym2})}(-1)^{|ab'|'|g|'} \langle  g, a b'  \rangle
 $$

and in our grading, where for all $a\in A, f\in A^*$, $|a|'=-1, |f|'=0$, we get
\begin{equation}\label{cycL}
\langle g a, b' \rangle = \langle a b', g \rangle
\end{equation}

and

\begin{equation}\label{symL}
\langle ab', g \rangle = - \langle g, a b' \rangle
\end{equation}

respectively, for all $a,b' \in A, \,\, g\in A^*$

\end{proof}


Now we are ready to express three terms of the Maurer-Cartan equation \ref{MC4gr} via the bracket.

\begin{align*}
\langle m_3(ab,f,c), g \rangle&\mathop{=}^{def} \langle  g\otimes f, \d c, ab\c \rangle;
\\
\langle m_3(a,bf,c), g\rangle&\mathop{=}^{def}  \langle g \otimes bf ,\d c,a\c \rangle\mathop{=}^{R}
\langle g\otimes f, \d c,a \c b \rangle;
\\
\langle a m_3(b,f,c), g\rangle&\mathop{=}^{cycl.m_2} - \langle m_3(b,f,c),ga \rangle\mathop{=}^{def} - \langle ga\oo f, \d c,b \c  \rangle \mathop{=}^{L}
 \langle g \oo f, a \d c,b \c \rangle
\\
\end{align*}


According to these  the Maurer-Cartan can be rewritten as

$$\langle  g\otimes f, \d c, ab\c \rangle - \langle g\otimes f, \d c,a \c b \rangle -
\langle g \oo f, a \d c,b \c \rangle$$


which is exactly the Leibniz identity:

$$
\d c, ab\c = \d c,a \c b + a \d c,b \c
$$

Now it remains to prove that the double bracket defined via the solution of the Maurer-Cartan (of type B) as

 $$\langle g\otimes f,\d b,a\c\rangle= \langle m(a,f,b),g \rangle
$$

 for all $a,b \in A, f,g \in A^*$, does satisfy the Jacobi identity.

The appropriate part of the Maurer-Cartan equation to consider is the part of arity 5, with inputs from
$A$,$A^*$, $A$, $A^*$ and $A$.

General  Maurer--Cartan in arity 5 reads:

$$
(-1)^0 m_3(m_3(x_1, x_2,x_3), x_4, x_5) + (-1)^{|x_1|'} m_3(x_1, m_3( x_2,x_3, x_4), x_5)
+ $$
$$(-1)^{|x_1|'+|x_2|'} m_3(x_1, x_2, m_3( x_3, x_4, x_5) )=0
$$

Applying this to the input $a,f,b,g,c$ from $A$,$A^*$, $A$, $A^*$, $A$ we get

$$m_3(m_3(a,f,b),g,c)+ (-1)^{|a|'} m_3(a,m_3(f,b,g),c) + (-1)^{|a|'} (-1)^{|a|'+|f|'}m_3(a,f,m_3(b,g,c))=0.$$

Thus from the Maurer-Cartan we have.

\begin{equation}\label{maca5}
m_3(m_3(a,f,b),g,c)- m_3(a,m_3(f,b,g),c)- m_3(a,f,m_3(b,g,c))=0.
\end{equation}

Since we are going to prove the double Jacobi identity:

$$
\d a,\d b,c\c\c_L+\tau_{123}\d b,\d c,a\c\c_L+\tau_{132}\d c,\d a,b\c\c_L=0.
$$

we need to express double commutators via the operations - solutions of the Maurer-Cartan equation.

\begin{lemma}\label{trb} For any $a,b,c \in A$ and $\alpha,\beta,\gamma \in A^*$ the from the definition \ref{br} it follows:

$$
\langle \alpha\otimes\beta\otimes\gamma, \d a,\d b,c\c\c_L\rangle
=\langle m_3(m_3(c,\gamma,b),\beta,a),\alpha \rangle
$$

\end{lemma}

\begin{proof}

\begin{align*}
\langle \alpha\otimes\beta\otimes\gamma, \d a,\d b,c\c\c_L\rangle&
=\langle \alpha\otimes\beta, \d a,\langle{\rm id}\otimes\gamma, \d b,c\c\rangle \c
\rangle
\\
&=\langle m_3(\langle {\rm id}\otimes\gamma, \d b,c\c\rangle,\beta, a), \alpha \rangle
\\
&=\langle m_3(m_3(c,\gamma,b),\beta,a),\alpha \rangle
\end{align*}

\end{proof}


Clearly (\ref{maca5}) is equivalent to
\begin{equation}\label{maca50}
\langle m_3(m_3(a,f,b),g,c),h\rangle - \langle m_3(a,m_3(f,b,g),c),h\rangle- \langle m_3(a,f,m_3(b,g,c)),h\rangle=0.
\end{equation}
for any $a,b,c\in A$ and $f,g,h\in A^*$.

By Lemma~(\ref{trb}), the first summand in (\ref{maca50}) is given by
\begin{equation}\label{t1}
\langle m_3(m_3(a,f,b),g,c),h\rangle=\langle h \oo g \oo f, \d c,\d b,a\c\c_L\rangle.
\end{equation}

We show now that the second term  in (\ref{maca50}) is expressed via double commutator as:

\begin{equation}\label{t2}
\langle m_3(a, m_3(f,b,g),c),h\rangle=-\langle  g \oo f \oo h, \d b,\d a,c\c\c_L\rangle.
\end{equation}

Indeed, using
cyclic invariance and graded symmetry of the inner product, we see

$$\langle m_3(a, m_3(f,b,g),c),h\rangle=
(-1)^{|a|'(|m(f,b,g)|'+|c|'+|h|')} \langle m_3(m_3(f,b,g),c),h),a\rangle=
$$

$$
(-1)^{|a|'(|m(f,b,g)|'+|c|'+|h|')}  (-1)^{|m(f,b,g)|'(|c|'+|h|'+|a|')} \langle m_3(c,h,a), m_3(f,b,g))\rangle=
$$

$$
(-1)^{|a|'(|m(f,b,g)|'+|c|'+|h|')}  (-1)^{|m(f,b,g)|'(|c|'+|h|'+|a|')} \cdot - (-1)^{|m(c,h,a)|'|m(f,b,g)|'}
\langle m_3(f,b,g), m_3(c,h,a) \rangle=
$$

$$
(-1)^{|a|'(|m(f,b,g)|'+|c|'+|h|')}  (-1)^{|m(f,b,g)|'(|c|'+|h|'+|a|')} \cdot - (-1)^{|m(c,h,a)|'|m(f,b,g)|'}
$$
$$(-1)^{|f|'(|b|'+|g|'+|m(c,h,a)|')}
\langle m_3(b,g, m_3(c,h,a)),f )\rangle=
$$

$$
(-1)^{|a|'(|m(f,b,g)|'+|c|'+|h|')}  (-1)^{|m(f,b,g)|'(|c|'+|h|'+|a|')} \cdot - (-1)^{|m(c,h,a)|'|m(f,b,g)|'}
$$
$$(-1)^{|f|'(|b|'+|g|'+|m(c,h,a)|')}
(-1)^{|b|'(|g|'+|m(c,h,a)|'+|f|')}
\langle m_3(g, m_3(c,h,a),f),b)\rangle=
$$

$$
(-1)^{|a|'(|m(f,b,g)|'+|c|'+|h|')}  (-1)^{|m(f,b,g)|'(|c|'+|h|'+|a|')} \cdot - (-1)^{|m(c,h,a)|'|m(f,b,g)|'}
$$
$$(-1)^{|f|'(|b|'+|g|'+|m(c,h,a)|')}
(-1)^{|b|'(|g|'+|m(c,h,a)|'+|f|')} (-1)^{|g|'(|m(c,h,a)|'+|f|'+|b|')}
\langle m_3(m_3(c,h,a)),f,b),g \rangle.
$$

Taking into account that in our grading $|m(f,b,g)|'=|f|'+|b|'+|g|'+1=0$ and
$|m(a,f,b)|'=|a|'+|f|'+|b|'+1=-1$ for all $a,b\in A, f,g\in A^*$ we see that the latter sign is '-', hence we get the required:

$$\langle m_3(a, m_3(f,b,g),c),h\rangle=-\langle  g \oo f \oo h, \d b,\d a,c\c\c_L\rangle,$$
since due to Lemma~\ref{trb}
$$\langle m_3(m_3(c,h,a)),f,b),g \rangle = \langle  g \oo f \oo h, \d b,\d a,c\c\c_L\rangle.$$

Now we consider the third term in (\ref{maca50}) and show that it is expressed via double commutator as:

\begin{equation}\label{t3}
\langle m_3(a,f, m_3(b,g,c))),h\rangle=-
\langle   f \oo h \oo g, \d a,\d c,b\c\c_L\rangle.
\end{equation}

Indeed, using cyclic invariance, we see

$$\langle m_3(a,f, m_3(b,g,c)),h\rangle=
(-1)^{|a|'(|f|'+|m(b,g,c)|'+|h|')} \langle m_3(f, m_3(b,g,c),h),a\rangle=
$$
$$
(-1)^{|a|'(|f|'+|m(b,g,c)|'+|h|')} (-1)^{|f|'(|m(b,g,c)|'+|h|'+|a|')}\langle m_3( m_3(b,g,c),h,a), f\rangle
$$

Taking into account signs in our grading, we see that the letter sign is "-", thus by \ref{trb}

$$
(-1)^{|a|'(|f|'+|m(b,g,c)|'+|h|')} (-1)^{|f|'(|m(b,g,c)|'+|h|'+|a|')}
\langle m_3( m_3(b,g,c),h,a), f\rangle
$$
$$ - \langle m_3( m_3(b,g,c),h,a), f\rangle=
- \langle   f \oo h \oo g, \d a,\d c,b\c\c_L\rangle.
$$

Thus \ref{maca50} can be rewritten as:

$$
\langle   h \oo g \oo f, \d c,\d b,a\c\c_L\rangle +
\langle g\oo  f \oo h , \d b,\d a,c\c\c_L\rangle +
\langle   f \oo h \oo g, \d a,\d c,b\c\c_L\rangle=0
$$

We see that permutations of functionals $f,g,h \in A^*$ in our formulas match with the permutation on the images of the bracket in the double  Jacobi identity:

$$
\tau_{(123)}(h \oo g \oo f)=g\oo f\oo h, \tau_{(132)}(h \oo g \oo f)= f\oo h\oo g.
$$

 Thus we get the required identity~\ref{Jac}:

 $$
 \d c \d b,a \c\c_L + \tau_{(123)} \d b \d a,c \c\c_L + \tau_{(132)}  \d a \d c,b \c\c_L=0.
 $$

By this the proof of the first part of the theorem is completed.

\vspace{5mm}

Now we prove the {\it second part of the theorem}, i.e. that  the double Poisson algebras are in {\it bijection} with the pre-Calabi-Yau algebras of type B, given by $A_{\infty}$-structures on $A \oplus A^*=(A \oplus A^*,  m=m_2^{(1)}+m_3^{(1)})$. We need to check that all remaining components of the Maurer-Cartan equations does not give any other identities, but the double bracket axioms.

Thus we consider identities which appears from $MC_4$ and $MC_5$ for all possible types of input. We will show that from $MC_4$ on type B operations we get only Leibniz identity, written in one of two forms:

$$\d a,bc \c=b \cdot \d a,c\c+ \d a,b \c \cdot c
$$

$$ \d bc, a\c = b \star \d c,a \c + \d b,a \c \star c
$$
where  $\cdot$ denotes the outer multiplication on $A-A$ bimodule $A\otimes A$:
$c\cdot (a\otimes b)=c a \otimes b, \,\,  (a\otimes b) \cdot c = a \otimes b  c$,
and $\star$ denotes the inner multiplication on $A-A$ bimodule $A\otimes A$:
$c \star (a\otimes b)= a \otimes c b, \,\,  (a\otimes b) \star c = a  c \otimes b.$
Whenever we have anti-symmetry identity: $\d a,b\c=-\d b,a  \c^{op}$, these two forms of Leibniz identity are equivalent.

Then we show that from all $MC_5$ equations on  type B variables, the only nontrivial identities we get are two copies of the double Jacobi identity.

Note, that if $m_4=0$, but all higher operations $m_k, \,\, k\geq 5$ are arbitrary, we get Leibnitz from $MC_4$ and Jacobi identity from $MC_5$, however we can get extra identities from $MC_6$, those which connect $m_2$ and $m_5$, from $MC_7$, those which connect $m_2$ and $m_6$,  $m_3$ and $m_5$, etc.
Thus we emphasize that the second part of the theorem holds only for the pre-Calabi-Yau-structures of type B, given by the $A_\infty$-structures  $(A\oplus A^*, m=m_2^{(1)}+m_3^{(1)}$.

Let us start with $MC_4$. We need to consider all possible inputs for the arity 4 Maurer-Cartan, each will give a separate equation. A priori there are $2^4$ of such inputs, however, we will see, that when we consider ternary operations $m_3$ on $A \otimes A^*$ with only nonzero components corresponding to  tensors of type B: $A^*\otimes A \otimes A^* \otimes A$ or $A\otimes A^* \otimes A \otimes A^* $, some terms of these equations vanish. We also take in account grading, chosen  on $R=A \oplus A^*.$ For example, having this grading means that for any element $f,g \in A^*$, $fg=0$, since $|fg|'=|f|'+|g|'+1$ because $\deg m_2=1$ in our $conv.1$, but $|f|'+|g|'+1=1$, hence $|fg|'=1$, which for element $fg \in A^*$ is possible only if $fg=0.$

For inputs from $A\times A \times A \times A$ $MC_4$ is trivial. Let  $a,b,c,d \in A$, then from $MC_4$ we have an equation:

$$ m_3(ab,c,d) \pm m_3(a, bc,d) \pm m_3(a,b,cd) \pm a m_3(b,c,d) + m_3(a,b,c)d=0.$$

All terms in this equation vanish since neither of them correspond to the B-component of the structure, that is one of the operations
$A^* \times A \times A^* \to A^*$ or $A\times A^* \times A  \to A$, so they are equal to zero for all the type B variables.

The vanishing of the equation will happen for the inputs:
$$A^* \times A \times A \times A, \,  A \times A \times A \times A^*; \,
A^* \times A^* \times A^* \times A, \, A\times  A^* \times A \times A^*, \,
A^* \times A \times A^* \times A^*, \, A\times  A^* \times A^* \times A^*;$$
$$A \times A^*  \times A^* \times A, \, A\times  A \times A^* \times A^*, \, A^* \times A^*  \times A \times A; \,
A^* \times A^* \times A^* \times A^*, \, A\times  A \times A \times A.$$

Let us give an argument just for one of those as an example, less trivial than the one above, where all inputs were from $A$. Let us consider the input $f,a,g,h$ from $A^* \times A \times A^* \times A^*$. The $MC_4$ will look like
$$m_3(fa,g,h) \pm m_3(f,ag,h) \pm m_3(f,a,gh) \pm f m_3(a,g,h) \pm m_3(f,a,g)h = 0$$
The terms $m_3(fa,g,h), m_3(f,ag,h)$ and $ f m_3(a,g,h)$ are vanishing, since they correspond to operations of secondary type, which are absent from the B-type solution we are looking for. The term  $m_3(f,a,gh)$ contains among the arguments element
$gh \in A^*\cdot A^*$, which is according to our grading and $conv. 1$ should have degree 1, and since the only nonzero graded components are of degree $-1$ and $0$, elements from $ A^*\cdot A^*$ are equal to zero. The last term $ m_3(f,a,g)h $ itself is an element from $ A^*\cdot A^*$, because whenever $m_3(f,a,g) $ is an operation of B-type, it takes value in $A^*$. Hence this last term is vanishing as well.


Vanishing of the equations for the other inputs from the above list can be shown analogously.

The following inputs will give  nontrivial equations:
$$A \times A \times A^* \times A, A \times A^* \times A  \times A;
A^* \times A \times A \times A^*, A \times A^* \times A \times A^*, A^* \times A \times A^* \times A.$$

To write down the corresponding equations in terms of the double bracket defined as above, we need in addition to the lemma~\ref{LR} on the outer multiplication of the $A-A$ bimodule $A\otimes A$ (and respectively the inner multiplication on the dual bimodule
$(A\otimes A)^*$)), the following lemma on the inner multiplication $\star$ on the $A-A$ bimodule $A\otimes A$ (and respectively the inner multiplication on the dual bimodule
$(A\otimes A)^*$)).

\begin{lemma}\label{LR*}
The following equalities hold:

$${\rm (R^*)} \quad \quad \langle  g  \otimes  f a, \d b, c\c \rangle= -\langle  g  \otimes  f, a \star \d b, c\c  \rangle
$$

$${\rm (L^*)} \quad \quad \langle  a g  \otimes  f, \d b, c\c \rangle= \langle  g  \otimes  f,  \d b, c\c \star a \rangle
$$

\end{lemma}

\begin{proof}

We use here Sweedler notations: $\d b, c\c = \sum b_i \otimes c_i  = \sum b' \otimes c''$.

(R*)
 $$\langle  g  \otimes f a, \d b, c\c \rangle=
 \sum\langle  g  \otimes  fa,  b'  \otimes  c''\rangle=
$$
$$
 \sum\langle  g, b'\rangle \langle  fa, c''\rangle \mathop
=- \sum\langle  g,  b'\rangle \langle  f, a c'' \rangle=
 $$
$$
- \langle  g  \otimes  f, \sum  b'  \otimes a c'' \rangle=
- \langle  g  \otimes f, a \star \d b, c\c \rangle
 $$

 We used here:
 $$\langle  f a, c'' \rangle \mathop{=}^{(\ref{cycL})} \langle  ac'', f \rangle \mathop{=}^{(\ref{symL})} - \langle  f, ac''\rangle $$

(L*)
 $$\langle a g  \otimes  f, \d b, c\c \rangle=
 \sum\langle  a g  \otimes  f\rangle \langle  b'  \otimes  c''\rangle=
$$
$$
 \sum\langle a g  \otimes  b'\rangle \langle  f  \otimes  c''\rangle \mathop{=}^{(\ref{cycR})}
 \sum\langle  g, b'a \rangle \langle  f,  c'' \rangle
 $$
$$
 = \sum \langle g\otimes  f, b'a \oo c'' \rangle=
 \langle  g  \otimes f,  \d b, c\c \star a \rangle
 $$

\end{proof}

Now we need to check that indeed from the remaining four inputs we get exactly one of the two mentioned above forms of the Leibniz identity.

We start with the input $a,f,b,c$ from $A \times A^* \times A \times A.$

$$m_3(af,b,c)+ (-1)^{|a|'} m_3(a,fb,c)+ (-1)^{|a|'+|f|'} m_3(a,f,bc) + a m_3(f,b,c) +
m_3(a,f,b)c = 0$$

Taking into account vanishing terms, and our grading, it reads, after pairing with arbitrary element $g\in A^*$.

$$- \l m_3(a,fb,c), g \r - \l m_3(a,f,bc), g \r + \l  m_3(a,f,b)c, g \r  = 0$$

Using the definition of the bracket, this means

\begin{equation}\label{e1}
\langle  g  \otimes fb,  \d c, a\c  \rangle - \langle  g  \otimes f,  \d bc, a\c  \rangle + \langle c g  \otimes f,  \d b, a\c  \rangle = 0
\end{equation}

For the third term we also used cyclic invariance with respect to $m_2$ and anti-symmetry:

$$\l  m_3(a,f,b)c, g \r  = (-1 )^{|m_3(a,f,b)|(|c|'+|g|')}
\l cg, m_3(a,f,b)\r = $$
$$ (-1 )^{|m_3(a,f,b)|'(|c|'+|g|')} \cdot - (-1 )^{|m_3(a,f,b)|'|cg|'}
\l  m_3(a,f,b), cg \r = \l  m_3(a,f,b), cg \r = \l cg \oo f, \d b,a \c \r.$$

Now each of the first and the last terms can be rewritten using lemma~\ref{LR*} as follows, to match the Leibniz rule (in the second form, written via inner multiplication).

$$\l g \oo fb, \d c,a \c \r \mathop{=}^{R^*} - \l g \oo f, b \star \d c,a \c \r $$

and

$$\l cg \oo f, \d b,a \c \r \mathop{=}^{L^*} - \l g \oo f,  \d b,a \c \star c \r. $$

Thus \ref{e1} reads

$$\l g \oo f, b \star \d c,a \c \r - \l g \oo f,  \d bc,a \c \r + \l g \oo f,  \d b,a \c \star c \r=0,$$
which means

$$ \d bc,a \c = b \star \d c,a \c + \d b,a \c  \star c.$$
Thus we got Leibniz identity, written via the inner product.

\vspace{5mm}

Consider now input $f,a,b,g$ from $A^*\times A\times A\times A^*$.

$$m_3(fa,b,g)+ (-1)^{|f|'} m_3(f,ab,g)+ (-1)^{|f|'+|a|'} m_3(f,a,bg) + (-1)^{|f|'} f m_3(a,b,g) +
m_3(f,a,b)g = 0$$

Which in type B and for our grading becomes:

$$ m_3(fa,b,g) +  m_3(f,ab,g) -  m_3(f,a,bg)   = 0$$

These three terms are elements from $A^*$, so after pairing with an arbitrary element $c\in A$ we get an equivalent equality:

$$\l m_3(fa,b,g), c \r + \l m_3(f,ab,g), c\r - \l  m_3(f,a,bg), c \r  = 0$$

The first term can be rewritten as

$$\l  m_3(fa,b,g),c \r  = (-1 )^{|fa|'(|b|'+|g|'+|c|')}
\l m_3(b,g,c), fa\r $$

$$ \mathop{=}^{(\ref{br})} \l fa \oo g, \d c,b \c \r.$$

The second:

$$\l  m_3(f,ab,g),c \r  = (-1 )^{|f|'(|ab|'+|g|'+|c|')}
\l m_3(b,g,c), fa\r $$

$$ \mathop{=}^{(\ref{br})} \l f \oo g,  \d c, ab \c \r.$$

The third:

$$-\l  m_3(f,a,bg),c \r  = - (-1 )^{|f|'}
\l m_3(a,bg,c), f\r $$

$$ \mathop{=}^{(\ref{br})} - \l f \oo bg,  \d c, a \c \r.$$

Thus we got

$$  \l fa \oo g,  \d c, b \c \r   +  \l f \oo g,  \d c, ab \c \r
      -\l f \oo bg,  \d c, a \c \r = 0.$$

Applying Lemma~\ref{LR} to the first and last term we get

$$  -\l f \oo g,  a \d c, b \c \r   +  \l f \oo g,  \d c, ab \c \r
      -\l f \oo g,  \d c, a \c b \r = 0.$$

This ensures the Leibniz identity (written via outer product).

\vspace{5mm}

Consider now input $a,f,b,g$ from $A\times A^*\times A\times A^*$.

$$m_3(af,b,g)+ (-1)^{|a|'} m_3(a,fb,g)+ (-1)^{|a|'+|f|'} m_3(a,f,bg) + (-1)^{|a|'} a m_3(f,b,g) + m_3(a,f,b)g = 0$$

Two terms $ m_3(a,fb,g)$ and $m_3(a,f,bg) $ which are not of type B vanishes and after pairing with $c \in A$ we have:

$$\l m_3(af,b,g), c \r - \l a m_3(f,b,g), c\r + \l  m_3(a,f,b)g, c \r  = 0$$

Using cyclic invariance:

$$\l m_3(af,b,g), c \r = (-1)^{|af|'(|b|'+|g|'+|c|')} \l m_3(b,g,c), af \r = (-1)^{|af|'(|b|'+|g|'+|c|')} (-1)^{|b|'(|g|'+|c|'+|af|')} \l m_3(g,c,af), b \r$$

$$(-1)^{|af|'(|b|'+|g|'+|c|')} (-1)^{|b|'(|g|'+|c|'+|af|')}
(-1)^{|g|'(|c|'+|af|'+|b|')} \l m_3(c, af, b), g \r$$

$$=- \l m_3(c, af, b), g \r = - \l g \oo af, \d b,c \c \r
\mathop{=}^{R} \l g \oo f, \d b,c \c a \r $$

The second  term:

$$-\l a m_3(f,b,g), c \r \mathop{=}^{cycl.m_2} - (-1){|a|'} \l  m_3(f,b,g) c, a \r  $$
$$\mathop{=}^{cycl.m_2} - (-1)^{|a|'}(-1)^{|m_3(f,b,g)|'}\l ca,  m_3(f,b,g) \r =$$

$$  - (-1)^{|a|'}(-1)^{|m_3(f,b,g)|'} \cdot  - (-1)^{|ca|'|m_3(f,b,g)|'}    \l m_3(f,b,g), ca \r = -  \l m_3(f,b,g), ca \r
$$

$$ = -  (-1)^{|f|'(|b|'+|g|'+|ca|')} \l m_3(b,g, ca), f \r =  -  (-1)^{|f|'(|b|'+|g|'+|ca|')}
   (-1)^{|b|'(|g|'+|ca|'+|f|')} \l m_3(g, ca, f), b \r$$

   $$=-  (-1)^{|f|'(|b|'+|g|'+|ca|')}
   (-1)^{|b|'(|g|'+|ca|'+|f|')}  (-1)^{|g|'(|ca|'+|f|'+|b|')} \l m_3( ca, f, b), g \r=$$

$$\l m_3( ca, f, b), g \r= \l g \oo f, \d b,ca \c \r.$$

The last term

$$\l m_3( a, f, b) g, c \r \mathop{=}^{cycl.m_2}  (-1)^{|m_3( a, f, b)|'(|g|'+|c|')}    \l gc, m_3( a, f, b) \r$$

$$=- (-1)^{|m_3( a, f, b)|'|gc|'} \l m_3( a, f, b), gc \r =  \l m_3( a, f, b), gc \r, $$

which by definition of the bracket is:

$$
\l m_3( a, f, b), gc \r = \l gc \oo f, \d b,a \c    \r \mathop{=}^{(L)} - \l g \oo f, c\d b,a \c    \r   $$

Combining the three terms we get the Leibniz identity (for outer product).

\vspace{5mm}

And finally, to ensure that from $MC_4$ we got nothing but Leibnitz identity, we have to consider the last non-trivial input:
$f,a,g, b$ from $A^*\times A\times A^*\times A$.

$$m_3(fa,g,b)+ (-1)^{|f|'} m_3(f, ag, b)+ (-1)^{|f|'+|a|'} m_3(f,a, gb) +  m_3(f,a,g)b +
(-1)^{|f|}f m_3(a,g,b) = 0$$

Which in type B and for our grading becomes:

$$ m_3(f,a,g)b + f m_3(a,g, b) -  m_3(f,a,gb)   = 0$$

These three terms are elements from $A^*$, so after pairing with an arbitrary element $c\in A$ we get an equivalent equality:

$$\l m_3(f,a,g)b, c \r + \l fm_3(a,g,b), c\r - \l  m_3(f,a,gb), c \r  = 0$$

The first term can be rewritten as

$$\l m_3(f,a,g)b, c \r  \mathop{=}^{(cycl.m_2)}(-1 )^{|m_3(f,a,g)|'(|b|'+|c|')}
\l bc, m_3(f,a,g)\r =\l bc, m_3(f,a,g)\r=$$

$$- (-1 )^{|m_3(f,a,g)|'|bc|'} \l m_3(f,a,g), bc \r =
 - (-1 )^{|m_3(f,a,g)|'|bc|'} (-1 )^{|f|'(|a|'+|g|'+|bc|'} \l m_3(a,g, bc), f \r   $$

$$= - (-1 )^{|m_3(f,a,g)||bc|'} (-1 )^{|f|'(|a|'+|g|'+|bc|'}   (-1 )^{|a|'(|g|'+|bc|'+|f|'}
\l m_3(g, bc,f), a \r
$$

 $$= - (-1 )^{|m_3(f,a,g)|'|bc|'} (-1 )^{|f|'(|a|'+|g|'+|bc|'}   (-1 )^{|a|'(|g|'+|bc|'+|f|'}
 (-1 )^{|g|'(|bc|'+|f|'+|a|'} \l m_3( bc,f,a), g \r $$

 $$=\l m_3( bc,f,a), g \r = \l g\oo f, \d a,bc \c \r
$$

The second:

$$\l f m_3(a,g,b),c \r  \mathop{=}^{cycl.m_2} (-1 )^{|f |'|(m_(a,g,b)|'+|c|')}
\l m_3(a,g,b)c, f\r $$

$$=(-1 )^{|f|'(| m_3(a,g,b)|'+|c|')} (-1 )^{| m_3(a,g,b)|'(|c|'+|f|')}
\l cf, m_3(a,g,b)\r  =$$

 $$(-1 )^{|f|'(| m_3(a,g,b)|'+|c|')} (-1 )^{| m_3(a,g,b)|'(|c|'+|f|')} \cdot
-(-1 )^{|cf|'| m_3(a,g,b)|'} \l  m_3(a,g,b), cf\r  =$$

$$\l  m_3(a,g,b), c f \r = (-1 )^{|a|'(|g|'+|b|'+|cf|'}  \l  m_3(g,b,cf), a \r = $$

$$(-1 )^{|a|'(|g|'+|b|'+|cf|'}  (-1 )^{|g|'(|b|'+|cf|'+|a|'}\l  m_3(b,cf,a), g \r =
-\l  m_3(b,cf,a), g \r$$

$$\mathop{=}^{(\ref{br})}- \l g \oo cf, \d a, b \c \r \mathop{=}^{R} - \l g \oo f, \d a, b \c c \r
$$

The third:

$$-\l  m_3(f,a,gb),c \r  = - (-1 )^{|f|'(|a|'+|gb|'+|c|')} \l m_3(a,gb,c), f\r =$$

$$- (-1 )^{|f|'(|a|'+|gb|'+|c|')} (-1 )^{|a|'(|gb|'+|c|'+|f|')} \l m_3(gb,c,f), a\r =$$

$$- (-1 )^{|f|'(|a|'+|gb|'+|c|')} (-1 )^{|a|'(|gb|'+|c|'+|f|')} (-1 )^{|gb|'(|c|'+|f|'+|a|')}
 \l m_3(c,f,a), gb\r =$$

$$\l m_3(c,f,a), gb\r \mathop{=}^{(\ref{br})} \l  gb\oo f, \d a,c \c \r
\mathop{=}^{L}-   \l  g\oo f, b \d a,c \c \r $$

Thus we got

$$  \l g \oo f,  \d a, bc \c \r  =  \l g \oo f,  \d a, b \c c \r
      +\l g \oo f, b \d a, c \c \r = 0.$$

This ensures the Leibniz identity (written via outer product).

By this we complete the proof of the second part of the Theorem~\ref{main}

\end{proof}


\section{Polyderivations from pre-Calabi-Yau structures induce a Poisson bracket on representation spaces}\label{rep}

Consider polyderivations, that is maps $A_1\otimes\cdots\otimes A_r\to A_1\otimes\cdots\otimes A_k$, which satisfy kind of Leibniz identities.

\begin{definition}\label{LL} Let $PolyDer (A^{\oo n}, M)$, for any $A^{\oo n}$-bimodule $M$,  be the space of polyderivations, that is linear maps $\delta:A\otimes ...\oo A\to M$
satisfying the Leibniz identity:
$$
   \delta(a_1\oo...\oo a_i'a_i''\oo...\oo a_n  = (1\oo...\oo a_i'\oo...\oo 1)            \delta(a_1\oo...\oo a_i''\oo...\oo a_n) + \delta(a_1\oo...\oo a_i'\oo...\oo a_n)(1\oo...\oo a_i''\oo...\oo 1)
$$

\end{definition}

\begin{definition}\label{AA} We call a solution of the Maurer-Cartan equation on $A\oplus A^*$ (and a corresponding linear map $\delta:A\otimes A\to A\otimes A$) a {\it restricted polyderivation}, $\delta \in RPolyDer(A^{\oo 2}, A^{\oo 2 })$,  if its
projection to $B$-component
is a polyderivation.
\end{definition}


Now we can see that  pre-Calabi-Yau structure, that is a cyclicly invariant $A_\infty$-structure on $(A\oplus A^*,  m=\sum\limits_{i=2, i\neq 4}^{\infty}m_i^{(1)})$,
which is a restricted polyderivation gives rise to the Poisson bracket on the space ${\rm Rep}(A,m)$. Moreover the bracket  is ${\rm GL_n}$ invariant.

This follows from

\begin{theorem}\label{dpb} Pre-Calabi-Yau structure
on $A=(A\oplus A^*,  m=\sum\limits_{i=2, i\neq 4}^{\infty}m_i^{(1)})$, which is additionally a restricted polyderivation
$\delta:A\otimes A\to A\otimes A$, $\delta\in RPolyDer(A^{\otimes 2},A^{\otimes 2})$ 
gives rise to the double Poisson bracket.
\end{theorem}

\begin{proof}
As we have shown  in Theorem~\ref{main}, one can construct a double Poisson brackets
from
a pre-Calabi-Yau structures of type $B$
 We now want to show that for any solution of $MC_5$, its projection to the
space of solutions of type $B$  is also a solution. Thus this projection will satisfy double Jacobi identity and thus create a double Poisson bracket from any pre-CY structure, which is a restricted polyderivation.

This comes from the consideration of Section~2 on the structure of equations arising from $MC_5$. We showed that  all equations but those which are entirely on operations of type $B$ contain in each
term at least one operation of type $A$. Hence if we replace in a given solution of the Maurer-Cartan equation all $Y$s (corresponding to operations of type $A$) by zero, we get
all equations with $Y$ in them automatically satisfied. System of equations arising from $MC_5$ turns into its restriction to those equations which are on operations of type
 $B$ only. The latter, as we know (Theorem~\ref{main}) under the
assignment
$$
\langle g \oo f, \d a,b\c\rangle:=\langle g,m_3(a,f,b)\rangle
$$
(which automatically has antisymmetry) coincides with the double Jacobi identity.

Unfortunately, the analogous procedure of projection of any solution onto the components of type $B$ does not work in the same way for $MC_4$. This is why
we additionally asked for our arbitrary pre-Calabi-Yau structure to be a restricted polyderivation. Now the type $B$ component of any Maurer-Cartan solution gives us a double
Poisson bracket.
\end{proof}

\begin{corollary}\label{PCY} Any pre-Calabi-Yau structure of an arbitrary associative algebra $A$, given by an $A_{\infty}$-structure on $(A\oplus A^*,  m=\sum\limits_{i=2, i\neq 4}^{\infty}m_i^{(1)})$, which is
a restricted polyderivation $\rho:A\otimes A\to A\otimes A$ on ternary level, induces a Poisson structure on representation spaces ${\rm Rep}(A,n)$, which is $Gl_n$-invariant.
\end{corollary}

\begin{proof}
It comes as a direct consequence of Van den Bergh's construction for double Poisson bracket \cite{VdB}, after the application of Theorem~\ref{dpb}.
\end{proof}

{\bf Remark.}
  In spite for general  pre-CY structure ($m_4$ non necessary zero),  we have from $MC_5$ not precisely the Jacobi identity, but only up to certain terms involving cyclic tensor $A \oo A^* \oo A \oo A^* \oo A$, the noncommutative Poisson structure
  still can be constructed and the Poisson structure
  on representation spaces induced (however it requiters some restrictions on $A$), as it was shown in \cite{IK2}.
  In case $m_4=0$ the situation is much more transparent, and it shows clearly how the correcting term for the Jacobi identity appear, so we consider it separately in this paper.  Moreover, if $m=m_2+m_3$ we have even one-to-one correspondence between pre-CY-structures and double Poisson brackets.

{\bf Remark.}
 Note that there was a considerable freedom in the choice of definition for the double bracket via the pre-Calabi-Yau structure (in spite it is in many ways defined by various features of the double bracket), but only the one presented in definition~\ref{br}, together with appropriate choice of the grading on $A\oplus A^*$ allowed to deduce the axioms of double bracket. By this choice we thus found an embedding of a double Poisson structures into pre-Calabi-Yau structures.

 \vskip3mm







\section{Acknowledgements}
We are grateful to IHES and MPIM where this work partially was done for hospitality and support.
This work is funded by the ERC grant 320974, and partially supported by the project PUT9038 and EPSRC grant EP/M008460/1/.








\vspace{5mm}


\vspace{35mm}

\scshape

\noindent   Natalia Iyudu

\noindent School of Mathematics

\noindent  The University of Edinburgh

\noindent James Clerk Maxwell Building

\noindent The King's Buildings

\noindent Peter Guthrie Tait Road

\noindent Edinburgh

\noindent Scotland EH9 3FD

\noindent E-mail address:\quad { n.iyudu@ed.ac.uk }\\


\scshape

\noindent Maxim Kontsevich

\noindent  Institut des Hautes \'Etudes Scientifiques

\noindent 35 route de Chartres

\noindent F - 91440 Bures-sur-Yvette

\noindent  E-mail address:\quad { maxim@ihes.fr}\\

\scshape

\noindent Yannis Vlassopoulos

\noindent  Institut des Hautes \'Etudes Scientifiques

\noindent 35 route de Chartres

\noindent F - 91440 Bures-sur-Yvette

\noindent  E-mail address:\quad { yvlassop@ihes.fr}\\



\end{document}